\documentclass{article}
\usepackage{amsfonts,amsmath,amsthm,amssymb}
\usepackage{graphics, epsfig}
\usepackage{color}
\usepackage{appendix}
\usepackage{ulem}
\usepackage[makeroom]{cancel}
\usepackage{fancyhdr}
\usepackage{centernot}
\usepackage{mathtools}
\usepackage{ stmaryrd }

 \usepackage[usenames,dvipsnames]{pstricks}
 \usepackage{pst-grad} 
 \usepackage{pst-plot} 
\allowdisplaybreaks

\let\TeXchi\chi
\newbox\chibox
\setbox0 \hbox{\mathsurround0pt $\TeXchi$}
\setbox\chibox \hbox{\raise\dp0 \box 0 }
\def\chi{\copy\chibox}

\newtheorem{proposition}{Proposition}[section]
\newtheorem{theorem}{Theorem}[section]
\newtheorem{definition}{Definition}[section]
\newtheorem{example}{Example}[section]

\newtheorem{remark}{Remark}[section]

\newtheorem{question}{Question}[section]

\numberwithin{equation}{section}
\numberwithin{theorem}{section}
\numberwithin{definition}{section}
\numberwithin{example}{section}
\numberwithin{proposition}{section}
\numberwithin{lemma}{section}
\numberwithin{remark}{section}
\DeclareMathOperator{\id}{id}

\DeclareMathOperator{\Spec}{Spec}
\DeclareMathOperator{\Spa}{Spa}
\newcommand\blfootnote[1]{%
  \begingroup
  \renewcommand\thefootnote{}\footnote{#1}%
  \addtocounter{footnote}{-1}%
  \endgroup
}
\begin{document}
\title{Conical calculus on schemes and perfectoid spaces via stratification}
\author
{Manuel Norman}
\date{}
\maketitle
\begin{abstract}
\noindent In this paper we show that, besides the usual calculus involving K\"ahler differentials, it is also possible to define conical calculus on schemes and perfectoid spaces; this can be done via a stratification process. Following some ideas from [1-2], we consider some natural stratifications of these spaces and then we build upon the work of Ayala, Francis, and Tanaka [3] (see also [4-5] and [18]); using their definitions of derivatives, smoothness and vector fields for stratified spaces, and thanks to some particular methods, we are able to transport these concepts to schemes and perfectoid spaces. This also allows us to define conical differential forms and the conical de Rham complex. At the end, we compare this approach with the usual one, noting that it is a useful \textit{addition} to K\"ahler method.
\end{abstract}
\blfootnote{Author: \textbf{Manuel Norman}; email: manuel.norman02@gmail.com
}
\section{Introduction}
The concept of scheme was introduced by Grothendieck in his well known treatise EGA (see, for instance, [6]). A scheme is a locally ringed space which can be covered by affine schemes, that is, by locally ringed spaces which are isomorphic to the spectrum of some ring (the spectrum can be turned into a locally ringed space using Zariski topology and a certain structure sheaf; see [7-8] for more details). More recently, Scholze defined in [9] the concept of perfectoid space, which is similar, from some points of view, to the notion of scheme. The idea is to assign to any perfectoid affinoid $K$-algebra $(R,R^+)$ a certain affinoid adic space, namely $\Spa(R,R^+)$, which is called 'affinoid perfectoid space'; then, we define a perfectoid space to be an adic space over the perfectoid field $K$ which is locally isomorphic to some affinoid perfectoid spaces. For more details, we refer to [9-12] and [19]. These two notions are defined in such a way that they "locally resemble" some kind of space: schemes locally resemble affine schemes, while perfectoid spaces locally resemble affinoid perfectoid spaces. Another well known concept (that in fact can be define via ringed spaces, as schemes) is the one of manifold, where we are used to consider differentiation and integration, as for $\mathbb{R}^n$ \footnote{In fact, the motivation of this paper is to show that also conical calculus (different from the "usual" one with K\"ahler differentials) can be considered on schemes and perfectoid spaces.} (which indeed is what a manifold locally looks like). Many generalisations of manifolds arise in a similar way. Another idea of this type was introduced by tha author in [1] (and then developed in other papers): a structured space locally resembles various kinds of algebraic structures. The theory of structured spaces is not necessary to read this paper; however, we will follow some ideas from [1-2] (which are entirely reported here) and we will apply them in order to obtain a stratification of schemes and perfectoid spaces, that is, we will show that there is a natural way to associate to these spaces a certain poset, which will then give us a poset-stratified space (see Definition 2.1.3 and Remark 2.1.9 in [3]). Then, building upon [3], we will define derivatives over these stratifications, and this will allow us to extend the notion to schemes and perfectoid spaces (actually, the same method can be applied to any kind of space which locally resembles other spaces).
\section{Stratification of schemes and perfectoid spaces}
We begin showing how we can stratify schemes and perfectoid spaces. Actually, a similar process can be applied to any notion of space which "locally resembles" some other space. Let $X$ be a scheme or perfectoid space, and consider some open \footnote{If not otherwise specified, when we say 'open covering' we always refer to an open covering w.r.t. the topology defined on the scheme or perfectoid space, and not to other topologies that will be defined later (this is why we will prefer to consider two topologies; see Remark \ref{Rm:2.1}).} covering $( X_p )_p$ by affine schemes or affinoid perfectoid spaces, that is, a collection of open affine schemes or affinoid perfectoid spaces such that $\bigcup_p X_p = X$. We define a map $h:X \rightarrow \mathcal{L}$ as in Section 4 of [1], that is, we define:
\begin{equation}\label{Eq:2.1}
h(x):= \lbrace X_t \in (X_p)_p : x \in X_t \rbrace
\end{equation}
Intuitively, this map measures "how dense" is a point belonging to the underlying set of a scheme or perfectoid space w.r.t. the chosen cover (the dependence on this cover may be removed in some cases; see Section 2.2). The collection $\mathcal{L}$ may be defined, as in [1], to be the "power collection" of $(X_p)_p$ without the empty sets, that is, the analogue of the power set, but with collection of sets, where we exclude the empty sets. Now the idea is to define, as in Section 4 of [1], the following preorder on $X$:
\begin{equation}\label{Eq:2.2}
x \leq y \Leftrightarrow h(x) \subseteq h(y)
\end{equation}
It is immediate to check that this is indeed a preorder, but it may not be a partial order. We define the abstract subset $X/ \sim$ of $X$, which clearly becomes a poset under the above $\leq$, as the quotient of $X$ by the following equivalence relation:
\begin{equation}\label{Eq:2.3}
x \sim y \Leftrightarrow h(x)=h(y)
\end{equation}
Now we generalise an idea in [2]: we consider a particular kind of poset-stratified spaces, namely the ones obtained using a scheme or a perfectoid space $X$ and the corresponding poset $X/ \sim$. More precisely, we recall (see, for instance, [13-16]) that a (poset-)stratified space is a structure $(X,X \xrightarrow{s} P)$, with $X$ topological space, $P$ poset endowed with the Alexandroff topology, and $s:X \rightarrow P$ continuous surjection (there are some other slightly different notions, but here we will consider this one). We can then define:
\begin{definition}\label{Def:2.1}
Let $X$ be a scheme or a perfectoid space. Define the poset $X/ \sim$ via the equations \eqref{Def:2.1}, \eqref{Eq:2.2} and \eqref{Eq:2.3}, and endow it with the Alexandroff topology. A structure $(X,X \xrightarrow{s} X/ \sim)$, where $s:X \rightarrow X/ \sim$ is a continuous surjection, and the underlying set $X$ is also endowed with a second chosen topology (e.g. the smallest topology generated by the collection $(X_p)_p \cup \lbrace X, \emptyset \rbrace$), is called a 'stratification of the scheme/perfectoid space'.
\end{definition}
\begin{remark}\label{Rm:2.1}
\normalfont The above definition needs some remarks. First of all, the smallest topology considered there is essentialy the same as the one in Example 1.1 in [1] (this will allow us to follow an analoguous proof to the one in [2], which will give us an interesting example of stratification). Moreover, we can either decide to drop the previous topology of $X$ or to endow $X$ with a second topology, turning it into a bitopological space (see [17]). In the latter case, we will always refer (when dealing with $X/ \sim$ and with $s$) to the second topology, e.g. the smallest one defined above, while we will always refer to the former when considering the open coverings.\\
Another important aspect to notice is that a stratification of a scheme (or perfectoid space) is indeed a particular case of stratified space, as it is clear by its definition.
\end{remark}
As announced in the above Remark, we now give an explicit example of stratification, which will be regarded as the 'standard one'. The proof of this example is almost the same as the one of Proposition 5.1 in [2]; we rewrite it here.
\begin{example}\label{Ex:2.1}
\normalfont Let $X$ be any scheme of perfectoid space, and consider some covering $(X_p)_p$. Define the map $s:X \rightarrow X/ \sim$ by $s(x):=[x]$, and suppose that the covering is such that $\sup_{x \in X} |h(x)| < \infty$ \footnote{Here, we define the cardinality in the same way as for sets. For instance, let $A$, $B$ and $C$ be three different (i.e. they are not equal to each other) sets, and let $\mathcal{C}$ be the collection containing these sets. Then, $|\mathcal{C}|=3$.}. If we let the second topology be the smallest one generated by the covering (see the example in Definition \ref{Def:2.1}), then $(X, X \xrightarrow{s} X/ \sim)$ is a stratification of the scheme (or perfectoid space). To prove this, we need to show that $s$ is continuous and surjective. Surjectivity is clear, since each $[x] \in X/ \sim$ is reached at least by $x$. To prove continuity, first notice that, by definition of Alexandroff topology, whenever $U$ is an open subset of $X/ \sim$ and whenever $t \in U$, $r \in X/ \sim$, we have $r \in U$. Now, consider $s^{-1}([x])$. This is equal to the set:
$$ s^{-1}([x])=\lbrace y \in X: h(y)=h(x) \rbrace $$
If we have, generally:
$$ h(x)=\lbrace X_t \rbrace_{\text{some} \, X_t \text{'s in} \, (X_p)_p} $$
then we know that all the $x \in s^{-1}([x])$ belong to $\bigcap_{X_t} X_t$, with the same $X_t$'s as before. Actually, we can even be more precise: we can delete the other $X_i$'s in $(X_p)_p$ that intersect $\bigcap_{X_t} X_t$ but do not belong to the collection of $X_t$'s above. This means that:
$$ s^{-1}([x]) = (\bigcap_{X_t} X_t) \setminus (\bigcup_{\substack{ X_i \, \text{not belonging} \\ \text{to the previous} \\ \text{collection of} \, X_t \text{'s}}} X_i) $$
If $U$ is an open subset of $X/ \sim$, by what we said above we have that, whenever $x \in U$, all the points $y \in X/ \sim$ such that $x \leq y$ belong to $U$. This implies, by definition of $\leq$, that the set $s^{-1}(U)$ is the union of some intersections of sets in $(X_p)_p$: indeed, the gaps due to the differences of sets are filled because $h(x) \subseteq h(y)$ and by the previous discussion. Consequently, $s^{-1}(U)$ can be written as a union of some intersections as the above one. But by assumption we have intersections of a finite number of elements $X_t$, which are open by definition of the (second) topology on $X$. Thus, these intersections are open, and the union of these open sets is open. This proves the continuity of $s$, and thus we have verified that $(X, X \xrightarrow{s} X/ \sim)$ is a stratification of the scheme (or perfectoid space) $X$.
\end{example}
Clearly, stratifications can be used to study schemes and perfectoid spaces from other perspectives. We only sketch some of the possible results, even though they are not needed in the rest of this paper. Section 2.2 contains instead an important and useful "refinement": using direct limits (whenever possible), we will avoid the dependence on the covering $(X_p)_p$.
\subsection{Some results on stratifications}
Consider any collection of objects taken from the category of schemes (or from the category of perfectoid spaces over some fixed $K$), and assign to each element of the family one and only one stratification $(X, X \xrightarrow{s} X/ \sim)$ (we maintain them fixed throughout the discussion). Then, consider the class of all these stratifications; they will be the objects in a new category, where the morphisms are all the usual ones between stratified spaces (not anymore the ones between schemes or perfectoid spaces). Notice that there is a bijection between the chosen collection and the new category, because each stratification also involves the scheme or perfectoid space itself.\\
This category is actually a full subcategory of \textbf{Strat} (see Section 4.2 in 14]), as it can be easily verified. Consequently, some results for \textbf{Strat} still hold also on this subcategory. In particular, it can be seen that the following result holds:
\begin{proposition}\label{Prop:2.1}
Every subcategory of \textbf{Strat} constructed as above and equipped with the class of weak equivalences is a homotopical category.
\end{proposition}
\begin{proof}
See Lemma 4.3.7 in [14].
\end{proof}
Thus, it is possible to construct a homotopical category starting from any collection of schemes or perfectoid spaces over the same $K$, in such a way that there is a bijection between its objects and the elements in the chosen family.\\
Another possible kind of stratification can be obtained slightly changing Definition \ref{Def:2.1}. First of all, here we will endow in some way (below, two possible methods are shown) the space $X$ with a partial order. Then (after endowing it with the Alexadroff topology), we will consider the poset $X$ instead of $X/ \sim$ (this is the slight generalisation of Definition \ref{Def:2.1}, which leads to other possible kinds of stratifications of schemes and perfectoid spaces; clearly, such a structure is still a stratified space). More precisely, we know that $X$ is a preordered set under $\leq$. It is possible to define a partial order $\preceq$ on $X$ in the following standard way:
$$ x \prec y \Leftrightarrow x \leq y \, \text{and not} \, y \leq x $$
$$ x \preceq y \Leftrightarrow x \prec y \, \text{or} \, x=y $$
Another possible method to obtain a preorder is to start with some stratification $(X, X \xrightarrow{s} X/ \sim)$ and to define the following preorder (see Construction 4.2.3 in [14]):
$$ x \leq_s y \Leftrightarrow sx \leq sy $$
which can be turned into a partial order as in the previous situation. In any case, we obtain some poset structure for $X$. Now consider the stratified space $(X, X \xrightarrow{i} X)$ (where $i$ denotes the identity map). Then we have the following result:
\begin{proposition}\label{Prop:2.3}
Let $X$ be a scheme or a perfectoid space, and consider the stratified space $(X, X \xrightarrow{i} X)$, where $X$ is endowed with a poset structure either via the partial order obtained from $\leq$ or via the partial order obtained from $\leq_s$ (for some other stratification of $X$ related to $s$). Then, the following three results hold:\\
(i) $(X, X \xrightarrow{i} X)$ is a fibrant stratified space;\\
(ii) the stratified geometric realisation of the nerve of $(X, X \xrightarrow{i} X)$ is a fibrant stratified space;\\
(ii) $(X, X \xrightarrow{i} X)$ and the stratified geometric realisation of its nerve are homotopically stratified spaces.
\end{proposition}
\begin{proof}
(i) is Example 4.3.4 in [14], while (ii) is Example 4.3.5 in the same paper. (iii) is obtained applying Theorem 4.3.29 in [14] to the previous two results.
\end{proof}
It is also possible to associate to schemes and perfectoid spaces other kinds of spaces, as shown below.
\begin{proposition}\label{Prop:2.4}
Let $X$ be a scheme or a perfectoid space and consider a stratification $(X, X \xrightarrow{s} X/ \sim)$ (or even the more general kind of stratification $(X, X \xrightarrow{i} X)$). Then, we can define the following spaces:\\
(i) the simplicial set $SS(X)$, whose $n$-simplices are given by \textbf{Strat}$(\parallel \Delta^n \parallel, X)$;\\
(ii) the prestream $(X, \leq|_{\bullet})$ (or $(X, \leq_s|_{\bullet}$);\\
(iii) the d-space $(X, d^{\leq|_{\bullet}} X)$ (or $(X, d^{\leq_s|_{\bullet}} X)$).
\end{proposition}
\begin{proof}
For (i), see Definition 7.1.0.3 in [16]. For (ii), see 5.1.7 in [14] (this prestream is simply obtained by restriction on each open subset $U$ of $X$). For (iii), see 5.1.11 in [14].
\end{proof}
This allows us to study schemes and perfectoid spaces also from the points of view of simplicial sets, streams and d-spaces. A stream is usually defined to be a particular kind of prestream; see, for instance, Definition 5.1.14 in [14] for Haucourt streams and Remark 5.1.19 in the same paper for Krishnan streams. We will not go deeper into these topics here.
\subsection{Avoiding the dependence on the covering}
The results in this subsection can also be applied to Section 2.1. The idea is to define a direct limit in order to avoid the dependence on $(X_p)_p$ in the construction of a stratification. Given a scheme or a perfectoid space $X$, consider some open covering $(X_p)_p$. A \textit{refinement} of such a cover is another open cover $(Y_t)_t$ of $X$ such that:
$$ (X_p)_p \subseteq (Y_t)_t $$
This means that $(Y_t)_t$ contains all the affine schemes in the covering $(X_p)_p$, with some possible additional affine schemes. Of course, the same argument holds for perfectoid spaces and affinoid perfectoid spaces. The reason why we do this is due to $h$: more affine schemes in the covering can give, in general, more interesting posets $X/ \sim$.
\begin{remark}\label{Rm:2.2}
\normalfont As we had already noticed, we remark again the the term 'open' refers here to the topology defined on the scheme or perfectoid space $X$, and not to other topologies (that is, not to the second topology, which actually at this time has not been defined yet). 
\end{remark} 
Now consider for each open covering, say $(X^r _p)_p$, the corresponding poset $X/ \sim_r$. Since posets form a category, we can define a direct limit as follows. It is clear that the following implication holds:
$$ (X^r _p)_p \subseteq (X^t _p)_p \Rightarrow X/ \sim_r \subseteq X/ \sim_t$$
(where we consider the same representatives of $X/ \sim_r$ also on $X / \sim_t$, except when not possible, of course). Thus, the set of all the open covers (by affine schemes or affinoid perfectoid spaces) for $X$ is an index set and the family of all the corresponding posets is indexed by it. We consider the inclusion morphisms:
$$\iota: X/ \sim_r \rightarrow X/ \sim_t$$
for $(X^r _p)_p \subseteq (X^t _p)_p$. It is clear that all the necessary conditions are satisfied, and we can thus consider (when it exists) the following direct limit, which avoids the dependence on the chosen cover:
$$ \lim_{\longrightarrow_{(X^r _p)_p}} X/ \sim_r $$
If this limit exists, it will be called the 'refined corresponding poset of the scheme/perfectoid space', and we will usually consider it instead of any other covering.
\begin{remark}\label{Rm:2.3}
\normalfont Direct limits involving open covers are also used, for instance, when defining the "refined" \v{C}ech cohomology, that is, \v{C}ech cohomology that does not depend on the chosen cover. However, in our case we consider a different approach: for reasons due to the definition of $h$ and hence of $X/ \sim$, here it is more interesting to consider a refinement to be a covering with more elements than the previous one, and with at least all the previous affine schemes or affinoid perfectoid spaces. Instead, with \v{C}ech cohomology we consider a refinement to be a cover whose elements are subsets of some other elements in the other cover (see, for instance, Chapter 10 in [22]). The limit above may not exist, and in such cases we will unfortunately have a dependence on the chosen covering. However, we will see later that the definition of derivative actually depends on the considered stratification (hence, also on the chosen map $s$, not only on the poset $X/ \sim$): this can be seen similarly to the dependence on the chosen direction for directional derivatives. Thus, the choice of the covering will not cause problems, since the "stratified derivative" will always depend on a sort of "direction" (in this case, the stratification).\\
We also note that this definition of refinement does not lead, in general, to a "degenerate" poset, that is, a poset which turns out to actually be $X$ itself: indeed, since the notion of 'open' depends on the chosen topology, this could only happen with a discrete topology. We will use the degenerate case in the definition of derivative, because in such situations it turns out to be really useful: it allows us to obtain a map between schemes or perfectoid spaces from a stratified map.
\end{remark}
Now that we have prepared the groud for the application of the work of Ayala, Francis and Tanaka [3], we briefly review the fundamental part of their paper which allows us to finally define derivatives on schemes and perfectoid spaces (and actually, as previously remarked, also on any kind of space to which the arguments of this section can be applied).
\section{Conical derivation on schemes and perfectoid spaces}
We start this section recalling the notion of derivative for stratified spaces defined in [3]. We will then define a map that assigns to each $f:X_1 \rightarrow X_2$ (maps between schemes or perfectoid spaces) a function from the chosen stratification of $X_1$ and the chosen stratification of $X_2$ (again denoted by $f$), we will then extend this function and derive it, and we will finally define the derivative of $f:X_1 \rightarrow X_2$ using the above map.\\
Following Section 3.1 in [3], consider some compact \footnote{Notice that this is not so restrictive when considering our particular case of stratifications of schemes and perfectoid spaces. Indeed, by Definition \ref{Def:2.1} we know that we can choose any possible second topology on $X$, so we just need to consider one for which $X$ is compact. For example, the topology in Example 1.1 in [1] can be often used, because in many cases $X$ turns out to be compact.} stratified space $X$, and consider the stratified space $\mathbb{R}^i \times \mathtt{C}(X)$, where the cone $\mathtt{C}(X) \rightarrow \mathtt{C}(P)$ is defined as in Definition 2.1.14 in [3]:
\begin{equation}\label{Eq:3.1}
\mathtt{C}(X):= * \coprod_{\lbrace 0 \rbrace \times X} \mathbb{R}_{\geq 0} \times X
\end{equation}
and
\begin{equation}\label{Eq:3.2}
\mathtt{C}(P):= * \coprod_{\lbrace 0 \rbrace \times P} [1] \times P
\end{equation}
The space $\mathbb{R}^i \times \mathtt{C}(X)$, which will be indicated by $U$, is composed by points that will be denoted by $(x,[y,z])$, with $(x,y,z) \in \mathbb{R}^i \times \mathbb{R}_{\geq 0} \times X$ \footnote{If $X= \emptyset$, $[y,z]= *$.}. Thanks to the following identification, where $TM$ denotes the tangent bundle of the manifold $M$:
$$ T \mathbb{R}^i \times \mathtt{C}(X) \cong \mathbb{R}^i _v \times \mathbb{R}^i \times \mathtt{C}(X) = \mathbb{R}^i _v \times U $$
(where the points are indicated by $(v,x,[y,z])$), we have a homeomorphism $\gamma: \mathbb{R}_{>0} \times T \mathbb{R}^i \times \mathtt{C}(X) \rightarrow \mathbb{R}_{>0} \times T \mathbb{R}^i \times \mathtt{C}(X)$ given by:
$$ (a,v,x,[y,z]) \xmapsto{\gamma} (a,av+x,x,[ay,z]) $$
$\gamma$ can also be seen as a map $\gamma_{a,x}$, as explained at the beginning of Section 3.1 in [3].\\
Before proceeding, we need to define:
\begin{definition}\label{Def:3.1}
A continuous stratified map $f$ between two stratified spaces $(X, X \xrightarrow{s_1} P_1)$, $(Y, Y \xrightarrow{s_2} P_2)$ is a commutative diagram of this kind:
$$ X \quad \rightarrow \quad Y $$
$$ \downarrow \qquad \qquad \downarrow$$
$$P_1 \quad \rightarrow \quad P_2$$
\end{definition}
Now consider a continuous stratified map $f$ between two compact stratified spaces such that $\mathtt{C}(P_1) \rightarrow \mathtt{C}(P_2)$ sends the cone point to the cone point. The restriction to the cone point stratum is denoted by $f|_{\mathbb{R}^i}$. The map $f_{\Delta}$ (see Definition 3.1.2 in [3]) is then given by:
$$ f|_{\Delta}:=\id_{\mathbb{R}_{>0}} \times f|_{\mathbb{R}^i} \times f$$
The importance of the previous two maps can be seen in Example 3.1.3 of [3]: these maps allow us to recover the usual definition of derivative in a particular case. This suggests Definition 3.1.4, which is the notion of derivative for stratified spaces we were looking for:
\begin{definition}[Derivative of stratified maps]\label{Def:3.2}
Let $(X, X \xrightarrow{s_1} P_1)$, $(Y, Y \xrightarrow{s_2} P_2)$ be two compact stratified spaces and let $f$ be a continuous stratified map between $\mathbb{R}^i \times \mathtt{C}(X)$ and $\mathbb{R}^j \times \mathtt{C}(Y)$. $f$ is continuously derivable along $\mathbb{R}^i$ (or, equivalently, $f$ is $C^1$ along $\mathbb{R}^i$), if $\mathtt{C}(P_1) \rightarrow \mathtt{C}(P_2)$ sends the cone point to the cone point and if there is a continuous extension (which, if it exists, is unique) $\widetilde{D} f$:
$$ \mathbb{R}_{\geq 0} \times T \mathbb{R}^i \times \mathtt{C}(X) \xmapsto{\widetilde{D} f} \mathbb{R}_{\geq 0} \times T \mathbb{R}^j \times \mathtt{C}(Y)$$
$$\uparrow \qquad \qquad \qquad \qquad \uparrow$$
$$ \mathbb{R}_{> 0} \times T \mathbb{R}^i \times \mathtt{C}(X) \xmapsto{\gamma^{-1} \circ f_{\Delta} \circ \gamma} \mathbb{R}_{> 0} \times T \mathbb{R}^j \times \mathtt{C}(Y) $$
The restriction to $a=0$ is denoted by $D f$. $D_x f$ is defined as the composition of the projection map onto the second term and the map from $\mathbb{R}^i _v \times \lbrace x \rbrace \times \mathtt{C}(X)$ to $\mathbb{R}^j _w \times \lbrace f(x,*) \rbrace \times \mathtt{C}(Y)$. For $n>1$, a map is continuously derivable along $\mathbb{R}^i$ $n$ times (or, equivalently, it is $C^n$ along $\mathbb{R}^i$) if it is continuously derivable along $\mathbb{R}^i$ and if $D f$ is continuously derivable $n-1$ times along $\mathbb{R}^i \times \mathbb{R}^i$. If $f$ is continuously derivable $n$ times along $\mathbb{R}^i$, $\forall n$, then $f$ is $C^{\infty}$ along $\mathbb{R}^i$ (or, equivalently, it is conically smooth along $\mathbb{R}^i$).
\end{definition}
We can now finally apply the previous definitions to our case. We first need to assign to each map between schemes or perfectoid spaces a certain map between their chosen stratifications. Actually, we will do this in two steps: we first assign, in a uniquely determined way, a continuous stratified map to our map of schemes or perfectoid spaces; then, we assign in a certain way another continuous stratified map, we derive it and we transform it into a map between schemes or perfectoid spaces. We then conclude noting a useful generalisation, which can be regarded as the actual way to define the derivative. Start with a map $f:X_1 \rightarrow X_2$ between schemes or perfectoid spaces. By definition, we can take the map between the underlying sets, denoted again by $f$. Stratify in some way the two spaces, say with $s_1$, $s_2$, respectively. Then we can write the following diagram, where we should find a map $g$ for which it is commutative, i.e. $g \circ s_1 = s_2 \circ f$:
$$ X_1 \qquad  \xrightarrow{f} \qquad X_2 $$
$$ \downarrow s_1 \qquad \qquad \downarrow s_2$$
$$X_1/ \sim_1 \quad \xrightarrow{g} \quad X_2 / \sim_2 $$
We have already fixed, as usual, the representatives of the equivalence classes of $X_i/ \sim_i$. We would like to do the following:
$$g(s_1(x))=s_2(f(x)) \rightsquigarrow g(y)=s_2(f(s^{-1} _1 (y)))$$
The unique problem is that $s^{-1} _1 (y)$ is a set. Of course, if $f$ is constant on each of these sets, then everything works properly, but this is really restrictive, so we will need to do something more. The idea is to choose a representative for each $s^{-1} _1 (y)$, and to consider a new diagram. More precisely, choose one and only one representative for each $s^{-1} _1 (y)$ (notice that these sets form a partition of $X$), which will be denoted by $t_y$. Now let $R^{s_1} _{X_1}$ denote the subspace of all the chosen representatives $t_y$, for $y \in X_1/ \sim_1$ ($R$ stays for 'representatives'). We can endow this space with the subspace topology, so that the restrictions of the maps will still be continuous. We have thus arrived at the following diagram:
$$ R^{s_1} _{X_1} \quad \, \xrightarrow{f} \qquad X_2 $$
$$ \downarrow s_1 \qquad \qquad \downarrow s_2$$
$$X_1/ \sim_1 \quad \xrightarrow{g} \quad X_2 / \sim_2 $$
where $f$ and $s_1$ are actually restricted to their new domain. Now, we can define $g$ uniquely (up to the choice of the representatives) as follows:
\begin{equation}\label{Eq:3.3}
g:=s_2 \circ f \circ s^{-1} _1
\end{equation}
which clearly assures the commutativity of the new diagram. Notice that this can be certainly done because $s_1$ is a bijection between $X_1/ \sim_1$ and $R^{s_1} _{X_1}$. Therefore, this function is a continuous stratified map associated to $f$. By Definition \ref{Def:3.2}, we need to "extend" this map to another domain. We will do this in the following way:\\
1) We consider $i=j$;\\
2) The extension of the continuous stratified map obtained before is the "obvious one", that is, the component in the cone of $R^{s_1} _{X_1}$ (respectively, $X_2$) is sent to the component of the cone of $X_1/ \sim_1$ (respectively, $X_2/ \sim_2$) as described in [3] (recall that the cone is a stratified space); the first component is simply $w$; the function on the second component (which comes from \eqref{Eq:3.1}), that is, the first component of the cone, is simply $a$ (with $a \geq 0$). Moreover, the map between the cones of the posets is essentialy the same as before, with the addition that the cone point is sent to the other cone point. Finally, the extension of $f$ is similar to the first ones above, i.e. the first two components are as above, while the third component is the obvious one \footnote{For this last extension, recall the notion of 'cone functor', which is the functor $\mathtt{C}:\textbf{Top} \rightarrow \textbf{Top}$ given by $\mathtt{C} f : \mathtt{C}(X) \rightarrow \mathtt{C}(Y)$, which is defined for continuous maps $f: X \rightarrow Y$ by $\mathtt{C} f ([x,y]):=[f(x),y]$.}.\\
This gives us another continuous stratified map, say $\widehat{f}$, as requested by Definition \ref{Def:3.2}. Indeed, it is not difficult to see that the diagram obtained is still commutative. We then consider, if it exists, the derivative $\widetilde{D} \widehat{f}$ (or its restriction to $a=0$, denoted by $D \widehat{f}$).
\begin{remark}\label{Rm:3.1}
\normalfont Since the definition of derivative also involves $\mathbb{R}_{\geq 0}$, $\mathbb{R}_{> 0}$ and $\mathbb{R}^i$, we will also have some functions on the corresponding components, as noted above. These can be, for instance, simply $w$ and $a$, as in the previous definition. However, there are also other cases, as we will see in an example below. This last possibility is the "generalisation" we were talking about: we allow any possible kind of extensions, but of course we will need to specify this dependence when dealing with derivatives obtained this way.
\end{remark}
We now need to assign to this map a function between schemes or perfectoid spaces. The standard way to do this is to consider everything discrete: in order to have a general way that always works, we endow the domain and the codomain of the derivative with the discrete topology. Then, both the domain and the codomain can be seen as schemes or perfectoid spaces. We will prove this for schemes; a similar argument can be used with perfectoid spaces. First of all, recall that any topological space can be given the structure of a locally ringed space; to do this, consider the sheaf of continuous real-valued functions on each open subset. In order to have a scheme, we need an open covering by affine schemes. We will show that each singleton $\lbrace x \rbrace$, which is open, is an affine scheme. This will imply the statement above. We will prove that each $\lbrace x \rbrace$ is isomorphic, as a locally ringed space, to $\Spec(\mathbb{R})$. Since $\mathbb{R}$ is not only a ring, but also a field, $\Spec(\mathbb{R})$ is a singleton (its only prime ideal is $\lbrace 0 \rbrace$, which is in fact its only element), and moreover we have:
$$\mathcal{O}_{\Spec(\mathbb{R})} (\emptyset)= \lbrace 0 \rbrace$$
$$\mathcal{O}_{\Spec(\mathbb{R})} (\lbrace 0 \rbrace)= \mathbb{R}$$
$\lbrace x \rbrace$ is a ring (whose additive and multiplicative identity coincide, and are $=x$) which is clearly isomorphic to $\Spec(\mathbb{R})$, since they are both singletons. It remains to find isomorphisms:
$$\phi_U : \mathcal{O}_{\Spec(\mathbb{R})}(U) \rightarrow \mathcal{O}_{\lbrace x \rbrace} (\rho^{-1} (U))$$
for each open $U$ in $\Spec(\mathbb{R})$, and where $\rho(x)= \lbrace 0 \rbrace$. When $U= \emptyset$, we clearly have the zero isomorphism between the two singleton rings. When $U= \lbrace 0 \rbrace$, we can take the isomorphism sending each $r \in \mathbb{R}$ to itself. The commutativity property required by the definition of morphism of locally ringed spaces is obviously satisfied, because of the trivial morphisms involved. Thus, $\lbrace x \rbrace$ is isomorphic to $\Spec(\mathbb{R})$ and the domain and codomain of the derivative can be seen as schemes. We define the $n$-th derivative of $f$ as the $(n-1)$-th derivative of $D f$ (viewed as a map between schemes/perfectoid spaces). A map that can be derived infinitely many times is called conically smooth along $\mathbb{R}^i$.
\begin{remark}\label{Rm:3.2}
\normalfont It is important to note that the $n$-th derivative of the stratified map obtained from some $f$ as above is, in general, different from the $(n-1)$-th derivative of the map $D f$ viewed as a map between stratified spaces.
\end{remark}
We conclude this section with some examples.
\begin{example}\label{Ex:3.1}
\normalfont Consider some scheme or perfectoid space $X$, and let $X/ \sim$ be its corresponding poset. Stratify this space via the standard $s$ (see Example \ref{Ex:2.1}; here we assume that we are dealing with some covering for which $s$ is continuous, as in such example), choose as representatives $t_y$ for the sets $s^{-1}(y)$ precisely the same as the representatives in $X/ \sim$, so that $R^s _X$ is the same set as $X/ \sim$, even though these spaces are endowed with different topologies ($R^s _X$ has the subspace topology, while $X/ \sim$ has Alexandroff topology). Consider any map $f:X \rightarrow R^s _X$ (where the codomain is endowed with the trivial topology, we consider the trivial covering via all its singletons and we stratify it via the identity) defined as follows:
$$f(t_y):= t_y$$
while $f$ can be defined in any way for the other points in $X$. Then, in order to obtain a continuous stratified map, we need to find a function $g$ for which the following diagram commutes:
$$R^s _X \quad \, \xrightarrow{f} \quad R^s _X$$
$$\qquad \downarrow s \qquad \qquad \downarrow s_1=i$$
$$\, X/ \sim \, \, \quad \xrightarrow{g} \, \quad X/ \sim$$
Indeed, the \textit{set} $R^s _X / \sim_1$ is the same as the \textit{set} $X/ \sim$. This is because the definition of $\sim_1$ together with the trivial covering of $R^s _X$ imply that $R^s _X / \sim_1$ is equal to $R^s _X$, which is thus equal (as a set) to $X/ \sim$, as we noted above.\\
Via \eqref{Eq:3.3}, we conclude that $g$ is the identity map. This statement follows from the fact that the restriction of $s$ to $R^s _X$ is clearly the identity, and the same holds true for $f$. Thus, we have:
$$R^s _X \quad \, \xrightarrow{i} \quad R^s _X$$
$$\downarrow i \qquad \quad \downarrow i$$
$$X/ \sim \, \, \xrightarrow{g} \, X/ \sim$$
from which we clearly have $g=i$, where $i$ denotes the identity. Consequently, extending the identity map in the obvious way (as previously outlined), we obtain again the identity, whose derivative (by Example 3.1.7 in [3]) is the identity map (which can be easily turned into a map between schemes or perfectoid spaces).
\end{example}
\begin{example}\label{Ex:3.2}
\normalfont For the meaning of 'extension' in this example, see also Remark \ref{Rm:3.1}. Making use again of Example 3.1.7 in [3], any map between schemes or perfectoid spaces which is extended, when defining the derivative, to something of the form 
$$\widehat{f}(x,y,z)=(k(x),y,\rho_y (z))$$
has the following expression for the derivative at $a=0$:
$$ D \widehat{f}(1,v,x,[y,z]) = (1,Dk_x (v), k(x), [x,\rho_0(z)]) $$
where $v$ is the same as $v$ in the identification after equation \eqref{Eq:3.2}. As usual, this map can be turned into a map between schemes or perfectoid spaces.
\end{example}
\begin{remark}\label{Rm:3.3}
\normalfont We remark that here we have considered only one possible kind of conical derivative. It is clear that, if we find other ways to assign a continuous stratified map to our $f$, then we can define the derivative analogously to what we did above. It seems that our map is quite natural, even though we often need to restrict the domain. It would be interesting to find some conditions which assure that $g$ actually exists without using such restriction, and then define the derivative in those cases (notice that, for instance, with the identity there would be no need to restrict the domain: $g=i$ would work properly, yelding as derivative the identity map, which in fact is conically smooth along $\mathbb{R}^i$). Another natural kind of derivative is, for example, the following one. Consider the stratification $(X_1, X_1 \xrightarrow{i} X_1)$, with the partial order obtained from the specialisation preorder on $X$; we have the diagram:
$$ X_1 \qquad  \xrightarrow{f} \qquad X_2 $$
$$ \downarrow i \qquad \qquad \qquad \downarrow s_2$$
$$X_1 \qquad \xrightarrow{g} \quad X_2 / \sim_2 $$
which is certainly commutative if we define
$$g:= s_2 \circ f$$
without using any restriction. Then we can proceed as above and evaluate $D f$. Notice that we do not have a dependence on the first stratification here, because it has been fixed.\\
Furthermore, there is also another way (which in some cases is actually better than the discrete one) to endow the derivatives with the structure of maps between schemes or perfectoid spaces. If we view the space $X$ as a subset of the domain and codomain of the derivative via isomorphism, then $X$ is disjoint from the difference of these spaces and $X$ itself, and we can use the extension topology (see [31] or Section 4 in [1]) to endow them with a natural extension of the first and second topologies on $X$. The problem is then that the new space obtained should be a scheme or a perfectoid space, and this is not guaranteed.
\end{remark}
\section{Conical vector fields and conical differential forms}
We now discuss conical vector fields and conical differential forms. The definitions needed for this section are quite involved, so to refer to [3] (and actually also to [4-5] and [18]) instead of writing them here. A stratification of a scheme or perfectoid space is called $C^0$ if the stratified space is $C^0$. It may be helpful to redefine the second topology on $X$ so that it becomes paracompact, in case it were not. The definition of conically smooth stratification is analogous. A conical vector field on some conically smooth stratification of a scheme or perfectoid space $X$ is an element of the vector space $\Theta(X)$ of parallel vector fields defined in Definition 8.1.2 of [3]. Flows of conical vector fields are defined below it. The idea of the definition of conical differential form is the same as the usual one (see, for instance, [23]). Since $\Theta(X)$ is a vector space, we can consider its dual, that is, the space $\Theta^* (X)$ consisting of conical covector fields. As covector fields are also called differential forms, conical covector fields are called conical differential forms. We define conical $k$-forms on $X$ as the elements of the space:
$$ (\bigwedge^k \Theta(X))^* $$
A notion of wedge product for conical forms can be defined as usual. We would also like to define a sort of 'exterior derivative'. For conical $1$-forms, this is not a problem: by Section 20 in [23], we can \textit{define} it via a formula that holds for the usual differential $1$-forms
\begin{equation}\label{Eq:4.1}
d \omega (X_1,X_2):= X_1 \omega(X_2) - X_2 \omega (X_1) - \omega([X_1,X_2])
\end{equation}
where the Lie bracket can be defined in the usual way. For $n>1$, we recall the following formula for the Lie derivative of differential forms, which will be used below:
$$ \mathcal{L}_X \omega(X_1, ..., X_n)=X(\omega(X_1,...,X_n)) - \sum_{i=1}^{n} \omega(X_1,...,[X,X_i],...,X_n) $$
Moreover, we have Cartan formula:
$$ \mathcal{L}_X \omega = i_X (d \omega) + d(i_X \omega) $$
If we define the interior product as usual, we can thus define the exterior derivative of conical $n$-forms, for $n>1$, as follows:
\begin{multline}\label{Eq:4.2}
d(i_X (\omega(X_1,...,X_n))):= X(\omega(X_1,...,X_n)) - \sum_{i=1}^{n} \omega(X_1,...,[X,X_i],...,X_n) \\
- i_X (d (\omega(X_1,...,X_n)))
\end{multline}
With $X$ variable in the space of vector fields, we obtain the desired definition inductively (recall that $n=1$ has already been defined). Some natural questions arise, and these stimulate further research:
\begin{question}\label{Qst:4.1}
\normalfont Is this the "best" definition possible?
\end{question}
This question may actually depend also on the following one:
\begin{question}\label{Qst:4.2}
\normalfont Is it possible to find local expressions of differential forms, as for manifolds?
\end{question}
Relatively to Question \ref{Qst:4.2}, we notice that it could be possible to locally define the exterior derivative via such local expressions. Moreover, a possible way to answer this question could be found considering the strata of parallel vector fields and applying the theory developed in [3-5] and [18].\\
We conclude this section noting that the above definition of exterior derivative allows us to define a new de Rham complex, which will be called 'conical de Rham complex':
\begin{theorem}[Conical de Rham complex]\label{Thm:4.1}
Analogously to the classical case, we can define a 'conical de Rham complex' via the above definition of exterior derivative. More precisely, we have:
\begin{equation}\label{Eq:4.3}
0 \rightarrow \Theta (X)^* \xrightarrow{d} {\Theta^2}(X)^* \xrightarrow{d} {\Theta^3}(X)^* \xrightarrow{d} ...
\end{equation}
where the spaces involved are the vector spaces
$$ {\Theta^n}(X)^* := (\bigwedge^n \Theta(X))^* $$
and where the cohomology groups are given by the quotients of closed forms by exact forms (as usual).
\end{theorem}
\begin{proof}
The fact that $d \circ d = 0$ can be easily proved starting from the definition above (it is only a long calculation). By this fact, it clearly follows that exact forms are closed (where these notions are defined similarly to the classical case). Then, we clearly have a complex which is analogous to de Rham complex, but involving conical forms.
\end{proof}
An important question naturally arises:
\begin{question}\label{Qst:4.3}
\normalfont Is there any connection between conical de Rham complex and K\"ahler-de Rham complex?
\end{question}
This could also give an answer to Question \ref{Qst:4.1}.
\section{A comparison between the two approaches}
We now briefly compare the usual method with K\"ahler differentials and the conical approach. It is clear that K\"ahler approach is a bit involved: it uses various concepts, while the definition of derivative along $\mathbb{R}^i$ is easier to state. Furthermore, the module of $R$-derivations is in general \footnote{There are some exceptions: for instance, when we consider polynomials (see [20] for some examples) we usually have simple calculations, as in classical calculus. Of course, the difficulty arises in more general cases. In such situations, the conical approach can be easier.} more difficult to compute (and this is one of the reasons why we consider K\"ahler differential forms, see also [20]), while the extension in Definition \ref{Def:3.2} can be simpler, since it is only \textit{one kind of} derivative of a function, and not the module of all the possible derivations \footnote{Even for higher order derivations (see, for instance, [29-30]) we have some structures, usually groups or algebras, which may be difficult to compute. Again, in these situations the conical approach is often easier, if we only want to evaluate derivatives.}. In fact, the introduction of derivations was due to the problem of finding a definition of derivative in commutative algebra: we then consider all the linear maps satisfying some classical properties of the derivative in real and complex calculus. Here, instead, we have a method to obtain a certain kind of derivative (which also generalises the usual one), and not all the possible derivations. We can therefore see these two approaches as two different ways of generalising derivatives, similarly, for instance, to the the various kinds of fractional derivatives which do not satisfy product rule (but which were born for some particular motivations); these are also accompanied by other notions of fractional derivatives which satisfy it (notably, conformable fractional derivative, which has been used in many situations since its introduction).\\
When it comes to the definition of conically smooth spaces and conical vector fields, even the conical approach becomes more involved. We think that, at least for what is known now, the conclusion is: whenever we want to consider only derivatives, the conical approach is usually simpler, and thus often the best possible. However, when it comes to differential forms, both conical and K\"ahler approaches give interesting results, and they both end up with two kinds of de Rham complexes. Hence, in this situation, we cannot say (at this time) what could be more useful, and in which cases. To answer this, further research is needed.\\
Summing up, we can say that these two approaches are different from the beginning: in one case, we consider all the possible derivations (and the module, or group/algebra, is in general difficult to evaluate), while in the other case we find an interesting and precise definition of derivative, which was born because of some motivations, and which is easier to compute. Furthermore, when we consider differential forms, both methods give rise to interesting results, and it is not known yet when it is better to use one method or the other one.
\section{Conclusion}
In this paper we have shown a conical approach to calculus on schemes and perfectoid spaces. Actually, the stratification method can also lead to other interesting results, not necessarily related to calculus (for some examples, see Section 2.1). Building upon [3], we have defined a notion of derivative which is simpler than the one involving the computation of a module, or a group/algebra in the higher order case, of all the possible derivations. Moreover, when we consider differential forms, both K\"ahler and conical approach give rise to interesting results, notably two kinds of de Rham complexes. Conical calculus thus turns out to be a really interesting and useful \textit{addition} to the usual K\"ahler method. Some directions for future works are given in Section 4.\\
\\
\begin{large}
\textbf{References}
\end{large}
\\
$[1]$ Norman, M. (2020). On structured spaces and their properties. Preprint (arXiv:2003.09240)\\
$[2]$ Norman, M. (2020). (Co)homology theories for structured spaces arising from their corresponding poset. In preparation\\
$[3]$ Ayala, D.; Francis, J; Tanaka, H. L. (2017). Local structures on stratified spaces. Advances in Mathematics, Volume 307, 903-1028\\
$[4]$ Ayala, D.; Francis, J.; Rozenblyum, N. (2018). Factorization homology I: Higher categories. Advances in Mathematics, Volume 333, 1042-1177\\
$[5]$  Ayala, D.; Francis, J.; Rozenblyum, N. (2019). A stratified homotopy hypothesis. J. Eur. Math. Soc. 21, 1071-1178\\
$[6]$ Grothendieck, A.; Dieudonn\'e, J. (1971). \'El\'ements de g\'eom\'etrie alg\'ebrique: I. Le langage des sch\'emas. Grundlehren der Mathematischen Wissenschaften (in French). 166 (2nd ed.). Berlin; New York: Springer-Verlag\\
$[7]$ Leytem, A. (2012). An introduction to Schemes and Moduli Spaces in Geometry. Master Thesis, University of Luxembourg\\
$[8]$ Hartshorne, R. (1977). Algebraic geometry. Graduate Texts in Mathematics, volume 52, Springer, Springer-Verlag New York\\
$[9]$ Scholze, P. (2012). Perfectoid spaces. Publ. Math. Inst. Hautes Études Sci. 116: 245-313\\
$[10]$ Bhatt, B. (2017). Lecture notes for a class on perfectoid spaces. Lecture notes\\
$[11]$ Cais, B; Bhatt, B; Caraiani, A.; Kedlaya, K. S.; Scholze, P.; Weinstein, J. (2019). Perfectoid Spaces: Lectures from the 2017 Arizona Winter School. Mathematical Surveys and Monographs, Volume 242. American Mathematical Soc.\\
$[12]$ Scholze, P. (2014). Perfectoid spaces and their Applications. Proceedings of the International Congress of Mathematicians-Seoul 2014 II, 461-486, Kyung Moon Sa, Seoul\\
$[13]$ Yokura, S. (2017). Decompostion spaces and poset-stratified spaces. Preprint (arXiv:1912.00339)\\
$[14]$ Nicotra, S. (2020). A convenient category of locally stratified spaces. PhD thesis, University of Liverpool\\
$[15]$ Krishnan, S. (2009). A convenient category of locally preordered spaces. Applied Categorical Structures 17.5, 445-466\\
$[16]$ Nand-Lal, S. J. (2019). A simplicial approach to stratified homotopy theory. PhD thesis, University of Liverpool\\
$[17]$ Kelly, J. C. (1963). Bitopological spaces. Proc. London Math. Soc., 13(3), 71-89\\
$[18]$ Ayala, D.; Francis, J.;  Tanaka, H. L. (2017). Factorization homology of stratified spaces. Sel. Math. New Ser. 23, 293-362\\
$[19]$ Scholze, P. (2013). Perfectoid spaces: a survey. Current developments in mathematics 2012, Int. Press, Somerville, MA, 193-227\\
$[20]$ Fonseca, T. J. (2019). Calculus on schemes - Lecture 1. Lecture notes, University of Oxford\\
$[21]$ Johnson, J. (1969). K\"ahler differentials and differential algebra. Annals of Mathematics, 89 (1): 92-98\\
$[22]$ Gallier, J.; Quaintance, J. (2019). A Gentle Introduction to Homology, Cohomology, and
Sheaf Cohomology, University of Pennsylvania\\
$[23]$ Tu, L. W. (2008). An Introduction to Manifolds, 2nd edition. Springer\\
$[24]$ Grothendieck, A. (1966). On the de rham cohomology of algebraic varieties. Publications Math\'ematiques de L'Institut des Hautes Scientifiques 29, 95-103\\
$[25]$ Grothendieck, A. (1968). Crystals and the De Rham cohomology of schemes. In: Dix Expos\'es sur la Cohomologie des Sch\'emas. North-Holland, 306-358\\
$[26]$ Arapura, D.; Kang, Su-Jeong. (2011). K\"ahler-de Rham cohomology and Chern classes. Comm. Algebra, 39(4): 1153-1167\\
$[27]$ Hartshorne, R. (1975). On the De Rham cohomology of algebraic varieties. Inst. Hautes Etudes Sci. Publ. Math., (45): 5-99\\
$[28]$ Fu, G.; Hal\'as, M.; Li, Z. (2011). Some remarks on K\"ahler differentials and ordinary differentials in nonlinear control systems. Systems and Control Letters, 60: 699-703\\
$[29]$ Vojta, P. (2007). Jets via Hasse-Schmidt derivations. Diophantine geometry, volume 4 of CRM Series, 335-361. Ed. Norm., Pisa\\
$[30]$ Gatto, L.; Salehyan, P. (2016). Hasse-Schmidt derivations on Grassmann Algebras, with applications
to Vertex Operators. IMPA Springer Monographs, no. 4\\
$[31]$ Steen, L. A.; Seebach, J. A. Jr. (1995). Counterexamples in Topology. Dover reprint of 1978 ed., Berlin, New York: Springer-Verlag

\end{document}